\documentclass[11pt]{amsart}
\usepackage{pstricks,pst-plot}

\definecolor{Black}{cmyk}{0,0,0,1}
\definecolor{OrangeRed}{cmyk}{0,0.6,1,0}            % half magenta only, full yellow
\definecolor{DarkBlue}{cmyk}{1,1,0,0.20}
\definecolor{myblue}{rgb}{0.66,0.78,1.00}
\definecolor{Violet}{cmyk}{0.79,0.88,0,0}
\definecolor{Lavender}{cmyk}{0,0.48,0,0}

\parskip=\smallskipamount

\newtheorem{theorem}{Theorem}[section]
\newtheorem{lemma}[theorem]{Lemma}
\newtheorem{corollary}[theorem]{Corollary}
\newtheorem{proposition}[theorem]{Proposition}

\theoremstyle{definition}
\newtheorem{definition}[theorem]{Definition}

\newtheorem{remark}[theorem]{Remark}

\newcommand{\Aut}{\mathop{{\rm Aut}}}

\numberwithin{equation}{section}

\begin{document}
\title[Hyperbolic laminations]{Solving $\overline\partial_b$ on hyperbolic laminations}
%
%    General info
%
\subjclass[2000]{}
\date{\today}

\subjclass[2000]{}
\date{\today}

%\author{J. E. Forn\ae ss}
%\author{E. F. Wold}

\maketitle

\centerline{ J. E. FORN\AE SS\footnote{The first author is supported by an NSF grant DMS-1006294.
2010 AMS classification. Primary: 32V20, 32W10; Secondary: 32F45} and E. F. WOLD}

\begin{abstract}
Let $X$ denote a compact set which is laminated by Riemann surfaces. We assume
that $X$ carries a positive CR line bundle $ L\rightarrow X$. The main result of the paper is that there exists a positive integer $s$ so that if $v$ is any continuous $(0,1)$ form with coefficients in $L^{\otimes s}$ there exists a continuous 
section $u$ of $L^{\otimes s}$ solving the equation  $\overline\partial_b u=v$.
\end{abstract}

\maketitle

\section{Introduction}

The Cauchy-Riemann equations or the
$\overline{\partial}$ equation are among  the most important tools in complex analysis. This is true in one complex dimension as well as in several complex variables. On CR manifolds one has similarly
the tangential Cauchy-Riemann equations. In this paper we will study the special case of compact CR manifolds which are Leviflat and foliated by Riemann surfaces.  
In this case the tangential Cauchy-Riemann equations reduce to the $\overline{\partial}$ equation on the individual leaves.
Since the manifolds are compact one cannot expect to solve the $\overline\partial$-equation for $(0,1)$-forms in general, and the natural thing is to consider sections of  
positive line bundles over the manifolds.   Then of course, by classical theory, we may solve $\overline\partial$ on each individual leaf  - the difficulty is to obtain transversal regularity, \emph{i.e.}, that the solutions vary nicely when you compare nearby leaves.  
Our main result is: 

\begin{theorem}\label{main}
Let $X$ be a compact hyperbolic Riemann surface lamination with a CR line bundle $L\rightarrow X$, and assume that $L$ is equipped with a positive metric $\sigma$.  Then there exists an integer $s\in\mathbb N$ such that 
for any continuous $(0,1)$-form $v$ with coefficients in $L^{\otimes s}$,  there exists 
a  continuous section $u$ of $L^{\otimes s}$ solving $\overline\partial_b u=v$.
\end{theorem}

In  Xiaoai Chai \cite{Chai}, the analogous result was proved in general for the equation $du/dx=v$ for arbitrary foliations by real curves.  \

We will prove a stronger version of this theorem,  for the special case of a suspension over a compact Riemann surface, and as an application we will prove the following: 

\begin{theorem}\label{furstenberg}
There exists a $\mathcal C^1$-smooth hyperbolic minimal Riemann surface lamination in 
$\mathbb P^5$
with uncountably many extremal closed laminated currents which are mutually disjoint.  
\end{theorem}

The theorem is proved by solving $\overline\partial_b$ with \emph{smooth} transverse regularity 
for suspensions (see (5.1)), and thereby obtaining an embedding theorem \'{a} la Kodaira, applied to a
suspension considered in \cite{FornaessSibonyWold}.  Related to such an embedding result, 
Ghys \cite{Ghys2} and Deroin \cite{Deroin} have shown that meromorphic functions and projective maps separate points on  these laminations (see also Gromov \cite{Gromov}).  

Theorem \ref{furstenberg} is in strong contrast to the situation in $\mathbb P^2$ where any such
lamination supports a unique normalized $\partial\overline\partial$-closed 
laminated current \cite{FornaessSibony}.  We will discuss this in Section \ref{embedding}.\

In a sequel to this paper, we will discuss further applications, and also non hyperbolic laminations, as well as laminations without positive CR bundles. \\

We next describe the plan of the paper.
In Section 2 we discuss the Kobayashi metric $K_x$ on leaves of a hyperbolic foliation $X$, 
and we give a new proof of the upper semi continuity of $K_x$.  
Then in Section 3 we show some useful facts for the unit disc. In Section 4 we discuss families of positive line bundles over the unit disc and prove continuity estimates for $\overline{\partial}$ with values in these line bundles. The line bundles are trivial but the metric varies quite strongly.  In Section 5 we will prove the main theorem for a specific example: the case of a suspension over a compact Riemann surface. 
The proof will give the main ideas for the general case, it will reveal the further need for the results in section 2, but the situation being somewhat simpler than the general case, we will not need section 4.   We also obtain stronger transverse regularity in this special case.  In Section \ref{embedding} we prove Theorem \ref{furstenberg} via an embedding result \'{a} la Kodaira.  Finally we prove the main theorem solving $\overline\partial_b$ in Section 7.

The Cauchy-Riemann equations have been discussed in the Levi flat Riemann surface case before by several authors, see \cite{Cao-Shaw-Wang}, \cite{Iordan-Matthey}, \cite{Ohsawa-Sibony}, \cite{Siu} and references therein.

\section{Hyperbolic laminations and the Kobayashi Metric}

We first define what we mean by a Riemann surface lamination. Let $X$ be a topological space
with an open cover $\{U_\alpha\}_{\alpha \in A}.$ We assume that for each $\alpha$ there is a homemorphism $\phi_\alpha: U_\alpha\rightarrow\mathbb D(z_\alpha) \times T_\alpha(t_\alpha)$ where $\mathbb D$ is the open unit disc in $\mathbb C$ and $T_\alpha$ is a metrizable topological space. Moreover the
maps $\phi_{\beta\alpha}:=\phi_\beta \circ \phi^{-1}_\alpha$ have locally the form
$\phi_{\beta\alpha}(z_\alpha,t_\alpha)= (z_{\beta\alpha}(z_\alpha,t_\alpha),t_{\beta\alpha}(t_\alpha))$
where the function $z_\beta$ is holomorphic as a function of $z_\alpha$ for fixed $t_\alpha.$

The sets $U_\alpha$ are called flow boxes. The sets $\mathcal L_{\alpha,t_\alpha}:=\phi_\alpha^{-1}(\mathbb D \times \{t_\alpha\})$ are called plaques and are homeomorphic to the unit disc. A nonempty subset $\mathcal L\subset X$ is called a leaf (of the lamination) if whenever $x\in\mathcal L\cap U_\alpha$ for some $\alpha$ then $\mathcal L$ contains the plaque in $U_\alpha$ containing $x$ and moreover $\mathcal L$ is minimal with respect to this condition. The set $X$ is then a disjoint union of leaves and for every $x$, the leaf through $x, \mathcal L_x,$ consists of all points in $X$ which can be joined to $x$ with a curve which is locally contained in a plaque.  A basis for a topology on a leaf $\mathcal L$ is given by proclaiming that 
each plaque in $\mathcal L$ is an open set, and that each set $U\cap\mathcal L$
is open, where $U$ is an open subset of $X$.  Then each leaf is a Hausdorff 
topological space, and each leaf has a natural structure of a Riemann surface inherited from the maps $\phi_\alpha.$  We say that a Riemann surface lamination 
is hyperbolic if each leaf is hyperbolic, \emph{i.e.}, it is universally covered by the unit disk.  \

Let $L\rightarrow X$ be a continuous complex line bundle.  We will say 
that $L$ is a complex line bundle on $X$ if it is defined by transition functions 
$f_{\alpha\beta}$ on $U_\alpha\cap U_\beta$, where $f_{\alpha\beta}$
is holomorphic along plaques.  By a smooth section of $L$ we will mean 
a continuous section which is smooth along the leaves.    
A weight $\sigma$ on $L$ will be a family 
of continuous functions $\sigma_\alpha$ on $U_\alpha$, smooth along the plaques, 
with $\sigma_\alpha = \sigma_\beta + 2\cdot\log|f_{\alpha\beta}|$ on $U_\alpha\cap U_\beta$.  We also assume that all partial derivatives of each $\sigma_\alpha$ vary 
continuously between leaves.  The weight $\sigma$ is said to be positive 
if each $\sigma$ is strictly subharmonic along the leaves.  \

For a Riemann surface lamination the notions of the tangent- and co-tangent bundle only have meaning along the leaves.  Considering these however, 
we have a natural definition of $(0,1)$-forms with coefficients in $L$, and also 
the $\overline\partial$-operator acting on sections along the leaves, denoted by $\overline\partial_b$.   A $(0,1)$-form is said to be smooth if, in local coordinates, it 
is continuous and smooth along plaques.   Note that if $v_t(z)$, $t\in\mathbb T$, is a family of continuous $(0,1)$-forms on the disk $\mathbb D$, continuous in the parameter $t$, and if $u_t$ is a family of $L^2$-functions solving $\overline\partial u_t=v_t$
in the week sense, and also $\|u_t-u_{t_0}\|_{L^2(\mathbb D)}\rightarrow 0$ as $t\rightarrow t_0$ for all $t_0\in\mathbb T$, 
then $u_t(z)$ is continuous in both variables.  This follows from the facts that the family of solutions given by 
the Cauchy integral has this property, that weakly holomorphic functions are holomorphic, and the $L^2$ to sup-norm estimate.

Let $X$ be a Riemann surface lamination, and assume that $X$ is equipped with 
a hermitian metric $\|\cdot\|_X$ along leaves, varying continuously also between leaves.  
For a point $x\in X$ let $\mathcal F_x$ denote the family of holomorphic maps 
$f:\mathbb D\rightarrow\mathcal L_x$ with $f(0)=x$.  For a map $f\in\mathcal F_x$
and $\zeta\in\mathbb D$ we let $f'(\zeta)$ denote the tangent vector $f_*(\zeta)(\partial/\partial\zeta)$.
The Kobayashi metric $K=K_X$ at a point $x\in X$ is defined by 
$$
K_x(x):=[\underset{f\in\mathcal F_x}{\sup}\{\|(f'(0))\|_X\}]^{-1}.
$$

Our approach to solving $\overline\partial_b$ on hyperbolic laminations will be to lift the problem to line bundles over the  
unit disk (the universal covers of the leaves), solve the $\overline\partial$-equations 
there according to a certain procedure, and then push the solutions back down.   
It will therefore be important to understand how the leaves distribute, and, moreover, that the universal covering maps vary regularly when we pass between leaves.  We will need the following result regarding the Kobayashi metric 
$K_x$:

\begin{theorem}\label{kmetric}
Let $X$ be a compact Riemann surface lamination, and assume that all leaves in $X$ are hyperbolic.  
Then $K_x$ is a continuous function on $X$.  
Moreover, if $x_j$ is a sequence 
of points in $X$ converging to a point $x_0\in X$, $v_j$ is a sequence 
of tangent vectors at the points $x_j$ converging to a nonzero tangent vector $v_0$ at 
$x_0$, and if $f_j:\mathbb D\rightarrow\mathcal L_{x_j}$ are the universal 
covering maps with $f_j(0)=x_j, f'_j(0)=\lambda_j\cdot v_j$, $\lambda_j>0$, 
then the sequence $f_j$ converges uniformly on compacts to the universal 
covering map $f_0:\mathbb D\rightarrow\mathcal L_x$ with $f(0)=x$ and $f'(0)=\lambda_0\cdot v_0, \lambda_0>0$.  
\end{theorem}

This result was proved by Candel \cite{Candel} (see also \cite{Ghys},\cite{Verjovsky}).
Recently, Dinh, Nguyen and Sibony \cite{DihnNguyenSibony} proved that 
$K_x$ is actually H\"{o}lder continuous.  
Lower semi-continuity is proved following Brody \cite{Brody}: we first obtain a 
strictly positive lower bound $c>0$ for $K$; otherwise we could produce 
a non-degenerate image of $\mathbb C$ in a leaf $\mathcal L$.   Knowing this we 
have that any holomorphic map $f:\mathbb D\rightarrow\mathcal L$ satisfies 
$|f'(\zeta)|\leq c\cdot (1/(1-|\zeta|^2))$.  Hence $\mathcal F:=\{Hol(\mathbb D,\mathcal\mathcal L_x):x\in X\}$ is sequentially compact, and so $K$ is lower semi-continuous.  
Assume for a moment that we also know that $K$ is upper semicontinuous, and note the following: if $x\in X$ is a point, and $f:\mathbb D\rightarrow\mathcal L_x$ is holomorphic, $f(0)=x$, then $f$ is a universal covering map if and only if $|f'(0)|=K_x^{-1}$; this follows from the Schwarz lemma and the fact than such map can be factored through the universal covering map.  Hence any convergent sequence 
of universal covering maps is a universal covering map, and so we easily obtain the last claim of the theorem.  \

We will give a new proof of the upper semi-continuity of $K_X$.  
In the case where $X$ is a complex manifold, the upper semi-continuity 
of $K_X$ is a well known 
theorem of Royden \cite{Ro}.  The crucial point in his approach
is to prove that if $f:\mathbb D\rightarrow X$
is an embedding and if $r<1$, then $f(\overline{\mathbb D}_r)$ admits a Stein neighborhood 
in $X$, so the strategy is not immediately applicable in the case of laminations.

\begin{theorem}\label{usc}
Let $X$ be any Riemann surface lamination.  
Then the $K_X$ is upper 
semi-continuous on $X$.  
\end{theorem}
\begin{proof}
Let $x\in X$ and let $f:\mathbb D\rightarrow\mathcal L_x$ be a holomorphic
map with $f(0)=x$ and $f'(0)\neq 0$.  Let $x_j$ be a sequence of points 
in $X$ converging to $x$.  We will show that for any $0<r<1$, the map 
$f_r:=f|_{\mathbb D_r}$ is the uniform limit of maps $f_j:\mathbb D_r\rightarrow\mathcal L_{x_j}$. \

Let $Z\subset{\mathbb D}_r$ denote the singular locus of $f$, \emph{i.e.}, the finite set of points where $f'$ vanishes.  
We will cover $\overline{\mathbb D}_r$ by a suitable finite increasing sequence $A_j$
of closed topological disks.   Each $A_j$ is obtained by defining 
$A_j:=A_{j-1}\cup B_j$ where $B_j$ is a closed topological disk for 
$j\geq 2$, and $A_1=B_1$.  
We want that   
\begin{itemize}
\item[1)] $C_j:=B_{j+1}\cap A_{j}\neq\emptyset$, 
\item[2)] $(\overline{A_{j-1}\setminus B_j})\cap (\overline {B_j{\setminus A_{j-1}}})=\emptyset$, 
\item[3)] $Z\cap C_j=\emptyset$, and
\item[4)] $f(B_j)$ is contained in a coordinate chart $U_j\subset X$ for each $j$. 
\end{itemize}
We make sure that each $C_j$ has an open neighborhood $\tilde C_j$ in $\mathbb D$
such that
\begin{itemize}
\item[5)] $f|_{\tilde C_j}$ is injective. 
\end{itemize} 

By choosing $A_1$ small enough there is an open neighborhood $W_1$
of $A_1$ in $\mathbb D$ such that $f:\overline W_1\rightarrow\mathcal L_x$
is the uniform limit of a sequence $f_j:\overline W_1\rightarrow\mathcal L_{x_j}$; 
simply lift $f$ within a flowbox.   We will proceed by induction.  \

Assume that we have found an open neighborhood $W_k$ of $A_k$, $k\geq 1$, 
such that $f:\overline W_k\rightarrow\mathcal L_x$ is the uniform limit 
of a sequence $f_j:\overline W_k\rightarrow\mathcal L_{x_j}$.  
Choose an open neighborhood $\widehat B_{k+1}$ of $B_{k+1}$
such that $f(cl(\widehat B_{k+1}))$ is contained in a coordinate chart.  
By possibly having to choose a smaller $\tilde C_k$ we may assume that 
$\tilde C_k\subset\subset W_k\cap\widehat B_{k+1}$.

According to \cite{Forstneric}, Theorem 4.1., there exist open neighborhoods 
$A_k'$, $B_{k+1}'$ and $C_k'$ of $A_k$, $B_{k+1}$ and $C_k$ respectively, $C_k'\subset A_k'\cap B_{k+1}'\subset\tilde C_k$, such that if $\gamma:\tilde C_k\rightarrow\mathbb D$ is a holomorphic map sufficiently close 
to the identity, then there exist injective holomorphic maps $\alpha:A_k'\rightarrow\mathbb D$
and $\beta:B_{k+1}'\rightarrow\mathbb D$ such that $\gamma=\beta\circ\alpha^{-1}$ on $C_k'$. 
Moreover, $\alpha$ and $\beta$ can be assumed uniformly close to the identity, 
depending on $\gamma$.  

Fix a flow box containing $f(cl(\widehat B_{k+1}))$, and let $g_j:\widehat B_{k+1}\rightarrow\mathcal L_{x_j}$ be 
the sequence of maps obtained by lifting $f:\widehat B_{k+1}$ to the leaf 
$\mathcal L_{x_j}$.  Then $\gamma_j:=g_j^{-1}\circ f_j\rightarrow id$
uniformly on $\tilde C_j$ as $j\rightarrow\infty$.  Let $\alpha_j,\beta_j$
be a sequence of splittings as alluded to above.   Then if we choose a small
enough open neighborhood $W_{k+1}$ of $A_{k+1}$ we have that 
the map 
$\tilde f_j$ defined as $\tilde f_j:=f_j\circ\alpha_j$ near $A_k$
and $\tilde f_j=g_j\circ\beta_j$ near $B_{k+1}$ are well defined on $\overline W_{k+1}$
and converges uniformly to the map $f$ as $k\rightarrow\infty$.  

\end{proof}

\section{Preparations for the analysis on families of  unit discs}\label{preliminaries}

In this section we discuss decompositions of the unit disc which reflects the way the disc covers leaves
of hyperbolic laminations, and we also estimate the deck transformations. This will be used in the next section to investigate the $\overline{\partial}$-	equation for data pulled back from the lamination.

\begin{definition}
Throughout this paper we let $\psi(\zeta)$ denote the function $\psi(\zeta)=\log (1-|\zeta|^2)$ defined 
on the unit disk $\mathbb D$ in $\mathbb C$.
\end{definition}

Note that the Poincar\'{e} metric on the disk is given by $P(\zeta)=e^{-\psi(\zeta)}|d\zeta|$, and recall that the Poincar\'{e} distance $d_P(0,a)$ between 
the origin and a point $a\in\mathbb D$ is given by $d_P(0,a)=\frac{1}{2}\log[(1+|a|)/(1-|a|)]$.

\begin{definition}
We let $A_n:=\{\zeta\in\mathbb D:1-(\frac{1}{2})^{n}\leq |\zeta|<1-(\frac{1}{2})^{n+1}\}$.
We let $\mathbb D(n)=\{\zeta\in\mathbb D: |\zeta|\leq1-(\frac{1}{2})^{n+1}\}$.
\end{definition}

Note that for all $a\in A_n$ we have that 
\begin{itemize}
\item[i)] $(\frac{1}{2})^{n+1}\leq 1-|a|^2\leq(\frac{1}{2})^{n-1}$,
\item[ii)] $2^{n-1}\leq e^{-\psi(a)}\leq 2^{n+1}$, and 
\item[iii)] $d_P(0,a)\leq n+2$. 
\end{itemize}

\begin{lemma}\label{psiaut}
Let $\phi\in\Aut_{hol}\mathbb D$ with $\phi(0)\in A_n$.  Then $e^{-\psi(\phi(\zeta))}\leq 2^{n+3}e^{-\psi(\zeta)}$ for all $\zeta\in\mathbb D$.
\end{lemma}
\begin{proof}
Write $\phi(\zeta)=e^{i\beta}(\zeta-a)/(1-\overline a\zeta)$ with $a\in A_n$.  
Then $\phi'(\zeta)=e^{i\beta}\cdot(1-a\overline a)/(1-\overline a\cdot\zeta)^2$, and so 
$|\phi'(\zeta)|\geq\frac{1}{4}(1-|a^2|)$ for all $\zeta\in\mathbb D$.  
By the Schwarz-Pick Lemma we also have that 
$|\phi'(\zeta)|=(1-|\phi(\zeta)|^2)/(1-|\zeta|^2)$, so we get that 
$1/(1-|\phi(\zeta)|^2)\leq 4/[(1-|a|^2)(1-|\zeta|^2)]$.  
\end{proof}

\begin{lemma}\label{cover}
Let $\phi\in\Aut_{hol}\mathbb D$ with $\phi(0)\in A_n$.  Then 
$$
\mathbb D_{(\frac{1}{2})^{n+k+3}}(\phi(0))\subset\phi(\mathbb D_{(\frac{1}{2})^k}).
$$
\end{lemma}
\begin{proof}
>From the previous proof we have that $|\phi'(\zeta)|\geq (\frac{1}{2})^{n+3}$.
\end{proof}

\begin{lemma}\label{derivativeupper}
Let $\phi\in\Aut_{hol}\mathbb D$ with $\phi(0)\in A_n$.  Then 
$$
|\phi'(\zeta)|\leq 2^{n+2}
$$
for all $\zeta\in\mathbb D$.  
\end{lemma}
\begin{proof}
We have $|\phi'(\zeta)|=(1-|a^2|)/|(1-\overline a\zeta)|^2\leq (1+|a|)/(1-|a|)$.
\end{proof}

\begin{lemma}\label{derivative}
Let $\phi\in\Aut_{hol}\mathbb D$ with $\phi(0)\in A_n$.  Then 
$$
|\phi'(\zeta)|\leq
(1-r)^{-2}(\frac{1}{2})^{n-1}
$$
for $\zeta\in\mathbb D_{r}$.
\end{lemma}
\begin{proof}
We have that $|\phi'(\zeta)|=|(1-|a^2|)|/|(1-\overline a\zeta)^2|$ for all $\zeta\in\mathbb D$.
\end{proof}

\begin{lemma}
Let $\phi\in\Aut_{hol}\mathbb D$ with $\phi(0)\in A_n$.  For any $0<r<1$ we 
have that 
$$
|\phi(\zeta)|\geq 1-(\frac{1}{2})^n\cdot [1 + \frac{2r}{(1-r)^2}], 
$$
for all $\zeta\in\mathbb D_r$.  
\end{lemma}
\begin{proof}
This follows from Lemma \ref{derivative} and the mean value theorem.  
\end{proof}

\begin{lemma}\label{derivativer}
Let $r>0$.  There exists a constant $c>0$ such that 
$$
e^{-\psi(\phi(\zeta))}\geq c\cdot 2^n,
$$
for $\phi\in\Aut_{hol}\mathbb D$ with $\phi(0)\in A_n$, 
for all $\zeta\in\mathbb D_r$ and all $n\in\mathbb N$.   
\end{lemma}
\begin{proof}
Use the previous lemma.  
\end{proof}

Let $S$ denote the strip $\{\zeta\in\mathbb C:0<Re(\zeta)< 1\}$.  For $k=0,1,...,2^5-2$  and $l\in\mathbb Z$ let $S_{k,l}$
denote the rectangle 
$$
S_{k,l}:=\{x+iy\in S:k\cdot(\frac{1}{2})^5<x<(k+2)\cdot (\frac{1}{2})^5 \mbox{ and } l\cdot(\frac{1}{2})^8<y< (l+2)\cdot(\frac{1}{2})^8\}.
$$
Choose a partition of unity $\alpha_{k,l}$ with respect to the cover $\{S_{k,l}\}$ of $S$ which is translation 
invariant in the $y$-direction, \emph{i.e.}, $\alpha_{k,l+j}(x,y)=\alpha_{k,l}(x,y-j\cdot(\frac{1}{2})^8)$. Let $c_p>0$ be a constant such that $\|\alpha_{k,l}\|_{\mathcal C^1(S_{k,l})}\leq c_p$
for all $k,l$. \

For $n=1,2,..$ let $f_n:S\rightarrow A_n$ denote the map 
$$
f_n(x,y)=(1-(\frac{1}{2})^n + x\cdot(\frac{1}{2})^{n+1})e^{2\pi i y(\frac{1}{2})^{n+1}}.
$$ 
For $n=1,2,...$ let $\tilde S_{k,l,n}=f_n(S_{k,l})$ for $k=0,...,2^5-2$ and $l=0,1,...,2^{n+9}-1$, 
and let $\tilde\alpha_{k,l,n}$ denote the function $\tilde\alpha_{k,l,n}=\alpha_{k,l}\circ f_n^{-1}$.
Then $\{\tilde\alpha_{k,l,n}\}$ is a partition of unity with respect to the cover $\{\tilde S_{k,l,n}\}$
of $A_n^\circ$.  Note that 
\begin{itemize}
\item[1.] any point $\zeta\in A_n^\circ$ is contained in at most four $\tilde S_{k,l,n}$'s,
\item[2.] if $a\in\tilde S_{k,l,n}$ and $\phi\in\Aut_{hol}\mathbb D$ satisfies $\phi(0)=a$, then 
$\tilde S_{k,l,n}\subset\phi(\mathbb D_{\frac{1}{2}})$ (Lemma \ref{cover}), 
\item[3.] there exists a constant $\tilde c_p>0$ such that $\|\tilde\alpha_{k,l,n}\|_{\mathcal C^1(\tilde S_{k,l,n})}\leq\tilde c_p\cdot 2^{n}$ for all $k,l,n$.
\end{itemize}
It follows from Lemma \ref{cover}, Lemma \ref{derivative}, and $3.$, that 
\begin{itemize}
\item[4.] there exists a constant $\tilde c>0$ such that 
if $a\in\tilde S_{k,l,n}$ and $\phi\in\Aut_{hol}\mathbb D$ satisfies $\phi(0)=a$, then 
$\|\tilde\alpha_{k,l,n}\circ\phi\|_{\mathcal C^1(\phi^{-1}(\tilde S_{k,l,n}))}\leq\tilde c$ for all $k,l,n$.
\end{itemize}
Let $\chi(x)$ be a decreasing function which is one on the interval $(0,\frac{1}{4})$ and which is 
zero on $(\frac{3}{4},1)$.    Let $\tilde\chi_n:=\chi\circ f_n^{-1}$.   We may assume that 
\begin{itemize}
\item[5.]  if $a\in\tilde S_{k,l,n}$ and $\phi\in\Aut_{hol}\mathbb D$ satisfies $\phi(0)=a$, then 
$\|\tilde\chi_{n}\circ\phi\|_{\mathcal C^1(\phi^{-1}(\tilde S_{k,l,n}))}\leq\tilde c$ for all $k,l,n$.
\end{itemize}

\begin{lemma}\label{distconst}
For any $r>0$ there exists a constant $c>0$ such that if $E\subset A_n, n\in\mathbb N,$ is a
set of points with $d_P(e_1,e_2)\geq r$ for all $e_1,e_2\in E$ with $e_1\neq e_2$, 
then $\sharp(E)\leq c\cdot 2^n$.
\end{lemma}
\begin{proof}
Fix $k\in\mathbb N$ such that the Poincar\'{e} radius of the disk $\mathbb D_{(\frac{1}{2})^k}$ is less than $\frac{r}{2}$.  By Lemma \ref{cover} we have that 
if $\phi\in\Aut_{hol}\mathbb D$ with $\phi(0)\in A_n$ then $D_{(\frac{1}{2})^{n+k+3}}(\phi(0))\subset\phi(\mathbb D_{(\frac{1}{2})^k})$.   Copy the construction of 
the cubes $S_{k,l}$ as above, but with sides of length $(\frac{1}{2})^{k+3}$
and $(\frac{1}{2})^{k+6}$ respectively.   Then the corresponding 
cubes $\tilde S_{k,l}$ have diameters less than $(\frac{1}{2})^{n+k+3}$, and a number
$2^{k+4-2}\cdot (2^{n+k+8}-1)$ of cubes is needed to cover $A_n$.

\end{proof}

\section{Families of line bundles over the disk}

Our approach to solve $\overline\partial_b$ on lamination will be to solve 
$\overline\partial$ for sections of positive line bundles over $\mathbb D$.  
Let $L\rightarrow\mathbb D$ be a line bundle with a positive metric $\sigma$. 
Since any line bundle over $\mathbb D$ is trivial, we may solve $\overline\partial$
using H\"{o}rmander: assume that $dd^c\sigma\geq c\cdot dV$, and let $v\in L^2_{(0,1)}(L,\sigma)$.  Then there exists $u\in L^2(L,\sigma)$ with $\overline\partial u=v$ and 
$$
\underset{\mathbb D}{\int\int}|u|^2e^{-\sigma}dV\leq\frac{1}{c}\underset{\mathbb D}{\int\int}|v|^2e^{-\sigma}dV.
$$
We need to study how these (canonical) solutions vary for certain families of line bundles over $\mathbb D$. \

Given an open set $U\subset\mathbb C$ 
we let $\|\cdot\|_{U,1}$ denote the norm 
$$
\|g\|_{U,1}:=\underset{\zeta\in U,s+t\leq 1}{\sup}\{|(\partial^{s+t}g)/(\partial x^s\partial y^t)(\zeta)|
\}, 
$$
defined for each $g\in\mathcal C^1(U)$.   
Note that if $\sigma_1$ and $\sigma_2$ are metrics on a line bundle $L\rightarrow U$, then the difference $\sigma_1-\sigma_2$ is a function on $U$.

Let $\{U_j\}_{j=1}^\infty$ be a locally finite cover of the disk $\mathbb D$, 
and let $\mathbb T$ be a topological space.  We shall consider 
families of line bundles over $\mathbb D$ parametrized by $\mathbb T$. 
A line bundle $L_t$ is given by a collection of transition functions 
$f_{t,i,j}\in\mathcal O(U_{ij})$, and a metric $\sigma_t$ on $L_t$
is given by a collection $\sigma_{t,j}$ of locally integrable functions, satisfying 
the compatibility condition
$$
\sigma_{t,i}-\sigma_{t,j}=2\cdot \log|f_{t,i,j}|
$$
on $U_{ij}$.  We will assume that 
\begin{itemize}
\item[1.] there exists a constant $c>0$ such that $dd^c\sigma_t\geq\frac{c}{(1-|\zeta|^2)^2}dV$ for all $t\in\mathbb T$,  
\item[2.] for any pair $i,j\in\mathbb N$ and any $t_0\in\mathbb T$ we have that $\|f_{t,i,j}-f_{t_0,i,j}\|_{U_{ij},1}\rightarrow 0$ as $t\rightarrow t_0$, and
\item[3.] for any $j\in\mathbb N$ and any $t_0\in\mathbb T$ we have that $\|\sigma_{t,j}-\sigma_{t_0,j}\|_{U_j,1}\rightarrow 0$ as $t\rightarrow t_0$.
\end{itemize}

\begin{remark}
By $3.$ it is understood that the non-smooth parts of the metrics cancel.  
\end{remark}

We may of course regard the union of the $L_t$-s as a bundle over $\mathbb D\times\mathbb T$. 
We denote this bundle by $L_{\mathbb T}$. \

\begin{remark}\label{invariantconstant}
Note that 
$$
dd^c(s\psi)(\zeta)=\frac{-4s}{(1-|\zeta|^2)^2}dV, 
$$
so if $c>4s$ we may solve $\overline\partial$ for sections in $L^2(L_t,\sigma_t+s\psi)$ with 
estimates: $dd^c(\sigma_t+s\psi)\geq (c-4s)dV$.  Also, if $\varphi\in\Aut_{hol}\mathbb D$ then 
$dd^c (\varphi^*\psi)=dd^c\psi$, and so $dd^c(\varphi^*\sigma_t + s\psi)\geq
(c-4s)dV$.  
\end{remark}

The following is the main result of this section. 

\begin{theorem}\label{continuous}
Let $\{L_t,\sigma_t\}_{t\in\mathbb T}$ be a family of line bundles satisfying 1.-- 3. Let 
$s\in\mathbb N$ and assume that $c>4s$.  Let $V\subset\subset\mathbb D$ be a domain, and
let $v_t\in \mathcal C_{(0,1)}(L_t,\sigma_t)$ be a continuous family of forms 
supported in $V$.  For each $t\in\mathbb T$ let $u_t$ be the $L^2(L_t,\sigma_t+s\psi)$-minimal solution to the equation $\overline\partial u_t=v_t$.  Then $u_t$ is a continuous 
section of $L_\mathbb T$.
\end{theorem}

We prove first some intermediate results, and then we prove the theorem at 
the end of the section .  

\begin{proposition}\label{debarondisk}
Let $\{L_t,\sigma_t\}_{t\in\mathbb N}$ be a family of line bundles satisfying 1--3,  let $s\in\mathbb N$ and 
assume that $c>4s$.  There exists a constant $c_1>0$ such that the following holds:
 
For any $t\in\mathbb T$ and for any section $u\in\mathcal OL^2(L_t|_{A_n},\sigma_t +s\psi)$, 
define $v:=\overline\partial(\tilde\chi_n\cdot u)$ ($v=0$ over $\mathbb D\setminus A_n$).
Then there exists $u_n\in \mathcal C^\infty L^2(L_t,\sigma_t+s\psi)$ with $\overline\partial u_n=v$, and 
$$
\underset{\mathbb D}{\int\int}|u_n|^2e^{-(\sigma_t+s\psi)}dV\leq c_1\underset{A_n}{\int\int}|u|^2e^{-(\sigma_t+s\psi)}dV.
$$
\end{proposition}

\begin{proof}
We use the partition of unity $\{\tilde\alpha_{k,l,n}\}$ with respect to $\{\tilde S_{k,l,n}\}$
defined in Section \ref{preliminaries}, and we write 
$$
v=\overline\partial(\tilde\chi_n\cdot u)=\sum_{k,l} \tilde\alpha_{k,l,n}\cdot\overline\partial(\tilde\chi_n\cdot u).
$$
Note that $v_{k,l,n}:=\tilde\alpha_{k,l,n}\cdot\overline\partial(\tilde\chi_n\cdot u)$ is $\mathcal C^\infty$-smooth
on $\mathbb D$ and is supported in $\tilde S_{k,l,n}$.

\begin{lemma}\label{debarsquare}
There exists a constant 
$c_3>0$, independent of $k,l,n$, such that the following holds:

There exists a section $u_{k,l,n}\in\mathcal C^\infty L^2(L_t,\sigma_t+s\psi)$
with $\overline\partial u_{k,l,n}=v_{k,l,n}$, and 
$$
\underset{\mathbb D}{\int\int}|u_{k,l,n}|^2e^{-(\sigma_t+s\psi)}dV\leq c_3\underset{\tilde S_{k,l,n}}{\int\int}|u|^2e^{-(\sigma_t+s\psi)}dV
$$
\end{lemma}
\begin{proof}
Let $\varphi\in\Aut_{hol}\mathbb D$ such that $\varphi(0)\in\tilde S_{k,l,n}$
and, consequently, 
$\tilde S_{k,l,n}\subset\varphi(\mathbb D_{\frac{1}{2}})$.
Let $v_{k,l,n}^*$ denote the section $v_{k,l,n}^*:=\varphi^*v_{k,l,n}$ of the bundle $\varphi^*L_t$.
We want to solve $\overline\partial u_{k,l,n}^*=v_{k,l,n}^*$, and then push the solution back forward using 
$\varphi$.   We use the metric $\varphi^*\sigma_t+s\psi$ (see Remark \ref{invariantconstant}).  
Note that, by 4. and 5. in Section \ref{preliminaries}, $v_{k,l,n}^*=\varphi^*\tilde\alpha_{k,l,n}\cdot\overline\partial [\varphi^*(\tilde\chi_n\cdot u)] = 
((\tilde\alpha_{k,l,n}\cdot u)\circ\varphi)\cdot\overline\partial [\varphi^*(\tilde\chi_n)] $, and so 
$|v_{k,l,n}^*|^2\leq c_4\cdot |\varphi^*u|^2$, where $c_4$ is independent of $k,l,n$ and $t$.  We have that 

\begin{align*}
\underset{\mathbb D}{\int\int}|v_{k,l,n}^*|^2e^{-(\varphi^*\sigma_t+s\psi)}dV & \leq 
c_4\underset{\varphi^{-1}(\tilde S_{k,l,n})}{\int\int}|\varphi^*u|^2e^{-(\varphi^*\sigma_t+s\psi)}dV\\
& =
c_4\underset{\tilde S_{k,l,n}}{\int\int}\varphi_*[|\varphi^*u|^2e^{-(\varphi^*\sigma_t+s\psi)}dV]\\
& = c_4\underset{\tilde S_{k,l,n}}{\int\int}|u|^2e^{-\sigma_t}\cdot e^{-2\psi}\cdot e^{-(s-2)(\psi\circ\varphi^{-1})}dV\\
& \leq c_5 \underset{\tilde S_{k,l,n}}{\int\int}|u|^2e^{-\sigma_t}\cdot e^{-2\psi}dV\\
& \leq c_6((\frac{1}{2})^n)^{s-2}\underset{\tilde S_{k,l,n}}{\int\int}|u|^2e^{-\sigma_t-s\psi}dV.
\end{align*}
By H\"{o}rmander there exists a section $u_{k,l,n}^*$
solving $\overline\partial u_{k,l,n}^*=v_{k,l,n}^*$ with 
$$
\underset{\mathbb D}{\int\int}|u_{k,l,n}^*|^2e^{-(\varphi^*\sigma_t+s\psi)}dV\leq (c-4s)^{-1}\cdot
c_6((\frac{1}{2})^n)^{s-2}\underset{\tilde S_{k,l,n}}{\int\int}|u|^2e^{-\sigma_t-s\psi}dV.
$$
Now let $u_{k,l,n}:=\varphi_*u_{k,l,n}^*$.  We get that   
\begin{align*}
\underset{\mathbb D}{\int\int}|u_{k,l,n}|^2e^{-(\sigma_t+s\psi)}dV & =
\underset{\mathbb D}{\int\int}\varphi^*[|u_{k,l,n}|^2e^{-(\sigma_t+s\psi)}dV]\\
& =  \underset{\mathbb D}{\int\int}|u_{k,l,n}^*|^2e^{-\varphi^*\sigma_t}\cdot e^{-2\psi}\cdot e^{-(s-2)(\psi\circ\varphi)}dV\\
& \leq (2^{n+3})^{s-2}\underset{\mathbb D}{\int\int}|u_{k,l,n}^*|^2e^{-\varphi^*\sigma_t}\cdot e^{-2\psi}\cdot e^{-(s-2)\psi}dV\\
& \leq (c-4s)^{-1}\cdot c_6\cdot 2^{3(s-2)}\underset{\tilde S_{k,l,n}}{\int\int}|u|^2e^{-\sigma_t-s\psi}dV,
\end{align*}
where in the first inequality we used Lemma \ref{psiaut}.
\end{proof}

By Lemma \ref{debarsquare} there exists for each pair $k,l$ a section $u_{k,l,n}$
solving $\overline\partial u_{k,l,n}=\tilde\alpha_{k,l,n}\cdot\overline\partial(\tilde\chi_n\cdot u)$, with 
$$
\underset{\mathbb D}{\int\int}|u_{k,l,n}|^2e^{-(\sigma_t+s\psi)}dV\leq c_3\underset{\tilde S_{k,l,n}}{\int\int}|u|^2e^{-(\sigma_t+s\psi)}dV.
$$
Define $u_n:=\sum_{k,l}u_{k,l,n}$.  Since any point $\zeta\in A_n$ intersects at most 
four $\tilde S_{k,l,n}$-s we get that 
$$
\underset{\mathbb D}{\int\int}|u_n|^2e^{-(\sigma_t+s\psi)}dV\leq 4\cdot c_3\underset{A_n}{\int\int}|u|^2e^{-(\sigma_t+s\psi)}dV.
$$
\end{proof}

\begin{corollary}\label{globalvsn}
There exists a constant $c_2$ such that the following holds.  Let $U\subset\subset\mathbb D$
and choose $N\in\mathbb N$ such that $A_n\cap\overline U=\emptyset$ for all $n\geq N$.
Let $v\in L_{(0,1)}^2(L_t,\sigma_t+s\psi)$ with $v$ supported in $U$.  Let $u_n$
be the $L^2(L_t|_{\mathbb D(n)},\sigma_t+s\psi)$-minimal solution to $\overline\partial u_n=v$,
and let $u$ be the $L^2(L_t,\sigma_t+s\psi)$-minimal solution to $\overline\partial u=v$.  We extend $u_n$ to $\mathbb D$ by setting $u_n=0$ outside $\mathbb D(n).$
Then 
$$
\|u_n-u\|_{L^2(L_t,\sigma_t+s\psi)}\leq c_2\cdot  \|u_n\|_{L^2(L_t|_{A_n},\sigma_t+s\psi)}.
$$
\end{corollary}
\begin{proof}
We let $\chi_n$ denote $\tilde\chi_n$ extended to be $1$ on $\mathbb D(n-1)$.  Let 
$\tilde u_n:=\chi_n\cdot u_n$ and let $\tilde u_n'=\tilde\chi_n\cdot u_n$ ($\tilde u_n'=0$ on $\mathbb D(n-1)$).  We have that $\overline\partial\tilde u_n=v + \overline\partial\tilde u_n'$.    Solve 
$\overline\partial u'=\overline\partial\tilde u_n'$ according to Proposition \ref{debarondisk}.  Then 
$\overline\partial (\tilde u_n - u')=v$.  If we let $\pi$ denote the orthogonal projection 
$\pi:L^2(L_t,\sigma_t+s\psi)\rightarrow\mathcal OL^2(L_t,\sigma_t+s\psi)^\perp$, we need 
to estimate $\|u_n - \pi(\tilde u_n - u')\|_{L^2(L_t,\sigma_t+s\psi)}$.  To simplify 
notation we denote the norm by $\|\cdot\|_t$.

We have 
\begin{align*}
\|u_n - \pi(\tilde u_n - u')\|_t& = 
\|\chi_n\cdot u_n + (1-\chi_n)\cdot u_n - \pi(\tilde u_n - u')\|_t\\
& \leq \|\tilde u_n - \pi(\tilde u_n)\|_t + \|(1-\chi_n)\cdot u_n\|_t + \|u'\|_t\\
& \leq\|\tilde u_n - \pi(\tilde u_n)\|_t + (1+\sqrt{c_1})\|u_n\|_{L^2(L_t|_{A_n},\sigma_t+s\psi)}
\end{align*}
Note that 
$$
\|\tilde u_n - \pi(\tilde u_n)\|\leq\underset{\|f\|_{L^2_t}\leq1}{\underset{f\in\mathcal OL^2_t}{\sup}}\{|\langle \tilde u_n,f\rangle|\}.
$$
Let $f\in\mathcal OL^2(L,\sigma_t+s\psi)$ with $\|f\|_{L^2(L,\sigma_t+s\psi)}=1$.
We have that $\langle (\chi_n+(1-\chi_n))u_n,f\rangle=0$, and so 
$|\langle\tilde u_n,f\rangle|=|\langle (1-\chi_n)u_n,f\rangle|$.  By the Cauchy-Schwartz inequality 
we get that
 $$
 |\langle\tilde u_n,f\rangle|\leq\|u_n\|_{L^2(L_t|_{A_n},\sigma_t+s\psi)}.
$$
\end{proof}

\begin{lemma}\label{closemetrics}
Let $U\subset\mathbb C$ be a domain, let $L\rightarrow U$ be a line bundle, and let $c>0$.
Then for any $\epsilon>0$ there exists a $\delta>0$ such that the following holds: Let $\sigma_1$
and $\sigma_2$ be metrics on $L$ with $dd^c\sigma_j\geq c\cdot dV$, 
and assume that $\|\sigma_1-\sigma_2\|_{U,1}\leq\delta$.  Let 
$v_i\in L^2_{0,1}(L,\sigma_i)$, for $i=1,2$, and let 
$u_i$ denote the $L^2(L,\sigma_i)$-minimal solution to the equation $\overline\partial u_i=v_i$ for 
$i=1,2$.  Then 
$$
\|u_1-u_2\|_{L^2(L,\sigma_1)}\leq c^{-1}\|v_1-v_2\|_{L^2(L,\sigma_1)}+\epsilon\|v_2\|_{L^2(L,\sigma_2)}.
$$
\end{lemma}
\begin{proof}
To shorten notation let $\|\cdot\|_j$ denote the $L^2$-norm with respect to the weight $\sigma_j$
for $j=1,2$.
Let $\pi_1$ denote the orthogonal projection $\pi_1:L^2(L,\sigma_1)\rightarrow\mathcal OL^2(L,\sigma_1)^\perp$.  Then $\pi_1(u_1-u_2)=u_1-\pi_1(u_2)$ satisfies 
$$
\|u_1-\pi_1(u_2)\|_{1}\leq c^{-1}\|v_1-v_2\|_{1}.
$$
Hence, we need to show that if $\delta$ is small enough, then $\|u_2-\pi_1(u_2)\|_{1}\leq\epsilon\|v_2\|_{2}$.  For this it is enough 
to show that if $g\in\mathcal OL^2(L,\sigma_1)$, $\|g\|_{1}\leq 1$, then $|\langle u_2, g\rangle_1|\leq\epsilon\|v_2\|_{2}$.
We have that 
\begin{align*}
|\langle u_2, g\rangle|_1 & = |\underset{U}{\int\int} u_2\cdot\overline ge^{-\sigma_1}dV|
= |\underset{U}{\int\int} u_2\cdot\overline {g\cdot e^{\sigma_2-\sigma_1}}e^{-\sigma_2}dV|.
\end{align*}
Let $\tilde v$ denote the form $\tilde v:=g\overline\partial(e^{\sigma_2-\sigma_1})$.  
Clearly, for any $\epsilon_1>0$, we may choose $\delta>0$ small enough such that 
$\|\tilde v\|_{2}\leq\epsilon_1$.  Let $\tilde u$ be the 
$L^2(L,\sigma_2)$-minimal solution to $\overline\partial\tilde u=\tilde v$.   We have that 
$$
\|\tilde u\|_{2}\leq c^{-1}\cdot\epsilon_1.
$$
We have that 
$$
\underset{U}{\int\int} u_2\cdot\overline {g\cdot e^{\sigma_2-\sigma_1} - \tilde u}\cdot e^{-\sigma_2}dV = 0,
$$
and so it is enough to estimate $|\langle u_2,\tilde u\rangle_2|$, and by Cauchy-Schwarz
we have that 
$$
|\langle u_2,\tilde u\rangle|_2\leq\|u_2\|_2\cdot c^{-1}\cdot\epsilon_1\leq c^{-2}\cdot\epsilon_1\cdot\|v_2\|_2.
$$
\end{proof}

\medskip

\emph{Proof of Theorem \ref{continuous}:}  

We may assume that $L_\mathbb T$ is the trivial bundle over $\mathbb D\times\mathbb T$
(solve Cousin II using the Cauchy integral formula for solving $\overline\partial$.) \

As stated, for each $t$ let $u_t$ denote the $L^2(\sigma_t+s\psi)$-minimal solution to the equation $\overline\partial u_t=v_t$.  Let $\epsilon>0$.  For each $n\in\mathbb N$ and $t\in\mathbb T$ let $u_{t,n}$ be 
the $L^2(\sigma_t + s\psi|_{\mathbb D(n)})$-minimal solution to the equation $\overline\partial u_{t,n}=v_{t}|_{\mathbb D(n)}$.  To simplify notation we 
denote the norms by $\|\cdot\|_{t}$ and $\|\cdot\|_{t,n}$.    
Fix $t_0\in\mathbb T$.
Then $u_{t_0,n}$ converges to $u_{t_0}$ in $L^2(\sigma_{t_0}+s\psi)$
and there exists an $N\in\mathbb N$ such that 
\begin{itemize}
\item[i)] $\|u_{t_0}-u_{t_0,n}\|_{t_0,n}\leq\epsilon$, and 
\item[ii)] $\|u_{t_0,n}\|_{L^2(\sigma_{t_0}+s\psi|_{A_n})}\leq\epsilon$,
\end{itemize}
for all $n\geq N$.   Fix $n_0\geq N$.  
By Lemma \ref{closemetrics} there 
exists an open neighborhood $V$ of $t_0$ such that 

\begin{itemize} 
\item[iii)] $\|u_{t,n_0}-u_{t_0,n_0}\|_{t,n_0}\leq\epsilon$,
\end{itemize}
for all $t\in V$.  By possibly having to choose a smaller $V$ we may assume that 
the any weight $\sigma_t$ is comparable to $\sigma_{t_0}$ on $\mathbb D(n_0)$, 
\emph{i.e.}, we have that 
$$
\frac{1}{2}e^{-\sigma_t}\leq 
e^{-\sigma_{t_0}}
\leq 2e^{-\sigma_t}
$$ 
for all $t\in V$.   We get that 
\begin{itemize} 
\item[iv)] $\|u_{t,n_0}\|_{L^2(\sigma_{t}+s\psi|_{A_{n_0}})}\leq 3\cdot\epsilon$,
\end{itemize}
for all $t\in V$.  By Corollary \ref{globalvsn} we have that 
\begin{itemize} 
\item[v)] $\|u_t - u_{t,n_0}\|_{t,n_0}\leq c_2\cdot 3\cdot\epsilon$,
\end{itemize}
for all $t\in V$.  
We get
\begin{align*}
\|u_t-u_{t_0}\|_{t_0,n_0} & \leq 2\cdot \|u_t-u_{t,n_0}\|_{t,n_0}\\
&+ 2\cdot\|u_{t,n_0}-u_{t_0,n_0}\|_{t,n_0}\\
& +\|u_{t_0,n_0}-u_{t_0}\|_{t_0,n_0}\\
& \leq (6\cdot c_2 + 3)\cdot\epsilon.
\end{align*}

\section{$\overline\partial_b$ on suspensions}

Our goal in this section is to prove Theorem \ref{main} in the special case that 
the lamination is a so-called suspension (see below for the construction). 
For applications we will show that we also get transversal smoothness, and we will allow singular metrics. 

\begin{theorem}\label{suspension}
Let $X$ be a compact Riemann surface of genus  greater than or equal to two, 
and assume that we are given a $\mathcal C^1$-smooth suspension $g:Y\rightarrow X$. 
Then there exists a constant $s>0$ such that the following holds:

Assume that we are given a line bundle $L_*\rightarrow X$ with a (possibly singular)
metric $\sigma_*$.
We let $L=g^*L_*$, 
$\sigma=g^*\sigma_*$, and we let $\tilde L$ denote the bundle $\pi^*L$
and $\tilde\sigma$ denote the metric $\pi^*\sigma$.   Let $\psi(\zeta):=\log(1-|\zeta|^2)$, 
and assume that $dd^c (\tilde\sigma + s\psi)$ is positive.   Then for 
any $\mathcal C^1$-smooth $(0,1)$-form $v$ on $Y$ with coefficients in $L$ and 
$v\in L^2_{loc}(\sigma)$,  there exists a $\mathcal C^1$-smooth section 
$u\in\Gamma(L)$ with $u\in L^2_{loc}(\sigma)$ and $\overline\partial_b u=v$.  \
To obtain transversally continuous solutions it is enough to assume $v$ is continuous and that $s=5$. 
\end{theorem}

\begin{remark}
A section/form being in $L^2_{loc}$ means that it is locally integrable in the leaf-direction
for each leaf. 
\end{remark}

\subsection{The construction of suspensions}

Let $X$ be a compact Riemann surface of genus greater than or equal to two (resp. one), let $f:\mathbb D\rightarrow X$ (resp. $f:\mathbb C\rightarrow X$)
be a universal covering map, and let $\Gamma$ be the corresponding 
Deck-group.  Let $\mathbb T$ be a compact smooth 
manifold, and assume that we are given a homomorphism
$\phi:\Gamma\rightarrow Diff(\mathbb T)$.  We let $\tilde\Gamma$
denote the group of diffeomorphisms of $\mathbb D\times\mathbb T$ (resp. $\mathbb C\times\mathbb T$)
consisting of elements $\tilde\varphi:=(\varphi,\phi(\varphi))$ for $\varphi\in\Gamma$, 
we consider the quotient $Y:=(\mathbb D\times\mathbb T)/\tilde\Gamma$
(resp. $Y:=(\mathbb C\times\mathbb T)/\tilde\Gamma$), and 
denote the quotient map by $\pi:\mathbb D\times\mathbb T\rightarrow Y$
(resp. $\pi:\mathbb C\times\mathbb T\rightarrow Y$).     \

For genus $g_X\geq 2$, coordinate charts on $Y$ are given as follows: for a point $(\zeta,t)\in\mathbb D\times\mathbb T$
let $U\subset\mathbb D$ be a domain such that $\varphi(U)\cap U\neq\emptyset,\varphi\in\Gamma\Rightarrow\varphi=id$.  Let $\tilde U:=\{[(\zeta,t)]:\zeta\in U,t\in\mathbb T\}$ and let 
$\Phi_{\tilde U}:\tilde U\rightarrow U\times\mathbb T$ be defined by $[(\zeta,t)]\mapsto (\zeta,t)$. 
Let $\tilde V$ be another chart with $\tilde U\cap\tilde V\neq\emptyset$.  Then 
there is a point $(\zeta_1,t_1)\in U\times\mathbb T$ and a point $(\zeta_2,t_2)\in V\times\mathbb T$ such that $[(\zeta_1,t_1)]=[(\zeta_2,t_2)]$, \emph{i.e.}, there is an element $\varphi\in\Gamma$
such that $\zeta_2=\varphi(\zeta_1)$ and $t_2=\phi(\varphi)(t_1)$.  So the transition 
$\Phi_{\tilde V,\tilde U}$ between $\Phi_{\tilde U}(\tilde U\cap\tilde V)$ and $\Phi_{\tilde V}(\tilde U\times\tilde V)$ is given by $(\zeta,t)\mapsto(\varphi(\zeta),\phi(\varphi)(t))$.  This gives $Y$ the 
structure of a Riemann surface lamination, and the leaves are the images $\pi(\mathbb D\times\{t\}), t\in\mathbb T$.  There is a natural projection $g:Y\rightarrow X$, given 
by $[(\zeta,t)]\mapsto [\zeta]$, and each fiber $Y_x:=g^{-1}(x)$ is diffeomorphic to $\mathbb T$.  
The lamination $Y$ is called a \emph{suspension} over $X$.   \

Now we want to define a transversal metric on $Y$ and describe a relationship 
with the Poincar\'{e} metric $d_P$ on $X$.   Let $\{U_j\}_{j=1}^m$ be a cover of $X$ 
by smoothly bounded disks.   We have charts 
$$
\Phi_j:g^{-1}(U_j)\rightarrow U_j\times\mathbb T, 
$$
respecting the projection to $U_j$.   Let $d_\mathbb T$ be any smooth 
Riemannian distance on $\mathbb T$, and for each $j$ let $d_j$ denote the transversal 
metric $d_j:=\Phi_j^*d$.   Note that any two distances $d_i$ and $d_j$ 
are comparable on a common domain of definition.  \

Let $\{\psi_j\}_{j=1}^m$ be a partition of unity with 
respect to the given cover of $X$, and define a
global transversal distance $d_x(t_1,t_2)$ by 
$$
d_x(t_1,t_2):=\sum_j \psi_j(x)d_j(t_1,t_2).
$$
For each $j$ and for each $i$ let $d_{ij}=(\Phi_j)_*d_i$.  
Then on $U_j\times\mathbb T$ we have that $(\Phi_j)_*d$
is given by 
$$
[(\Phi_j)_*d]_x(t_1,t_2)=\sum_i \psi_i(x)\cdot d_{ij}(t_1,t_2).
$$
For each $j$, let $C_j>0$ be a constant such that the following holds: if $x$ and $y$ 
are points in $U_j$, $t_1,t_2\in\mathbb T$, and $\gamma$ is a smooth curve connecting $x$ and $y$, then 

$$
[(\Phi_j)_*d]_y(t_1,t_2)\leq C_j^{l_P(\gamma)}\cdot [(\Phi_j)_*d]_x(t_1,t_2), 
$$
where $l_P$ denote the Poincar\'{e} length.  Choosing a constant $C$
which is greater than $C_j$ for all $j$ we obtain:  

\begin{lemma}
Given a $\mathcal C^1$-smooth suspension $g:Y\rightarrow X$, and a transversal metric $d_x$ as described above, there exists a constant $C>0$ such that the following holds: 

Let $x\in X$, 
let $t^x_1,t^x_2\in\mathbb T_x$, and let $\gamma: [0,1]\rightarrow X$ 
be a smooth immersion with $\gamma(0)=x$.   Let $y=\gamma(1)$,
and for $j=1,2$ let $t^y_j\in\mathbb T_y$ be the point obtained by lifting 
$\gamma$ to the leaf $\mathcal L_{t_j^x}$ with initial point $t_j^x$, \emph{i.e.},
$t^y_j$ is its end point.   Then 
$$
d_y(t_1^y,t_2^y)\leq C^{l_P(\gamma)}\cdot d_x(t_1^x,t_2^x).  
$$
\end{lemma}
\begin{lemma}\label{lenghtlift}
Given a $\mathcal C^1$-smooth suspension $g:Y\rightarrow X$ there exists a constant $k\in\mathbb N$
such that the following holds: 

If $\varphi\in\Gamma$ satisfies $\varphi(0)\in A_n$ then 
$$
d_0(\phi(\varphi)^{-1}(t^0_1),\phi(\varphi)^{-1}(t^0_2))\leq 2^{kn}\cdot d_0(t^0_1,t^0_2).
$$ 
for all points $t_1^0,t_2^0\in\mathbb T$.
(Here $d_0$ is the transversal metric constructed above lifted to $\mathbb D\times\mathbb T$ and restricted to $\{0\}\times\mathbb T=:\mathbb T_0$.)
\end{lemma}
\begin{proof}
Choose $k$ such that $2^k\geq C$ from the previous lemma.  Write $y=\varphi(0)$.
The points $t_1^y=\phi(\varphi)^{-1}(t^0_1)$ and $t_2^y=\phi(\varphi)^{-1}(t^0_2)$
are the points that are identified with $t_1^0$ and $t_2^0$ respectively 
by the map $\phi(\varphi)$, \emph{i.e.}, $(y,t_1^0)\sim (0,t_1^y)$ and $(y,t_2^0)\sim (0,t_2^y)$.   One way to locate the points $t_j^y$ is then to project 
the points $(y,t_j^0)$ to $Y$ by $\pi$ and the lift them back to $\mathbb T_0$.  \

Let $\tilde\gamma:[0,1]\rightarrow\mathbb D$ parametrize 
the straight line segment between $0$ and $y$, and let 
$\tilde\gamma_1(t)=(\tilde\gamma(t),t_1^0)$, $\tilde\gamma_2(t)=(\tilde\gamma(t),t_2^0)$.  The Poincar\'{e} length of $\tilde\gamma$ is 
less than $2^{n+2}$.   Projecting 
these curves to $Y$ and using  
the previous lemma we see that the transversal distance between the two points 
$\pi((y,t_j^0))$ is less than $2^{k(n+2)}\cdot d_0(t_1^0,t_2^0)$.  
\end{proof}

\subsection{Proof of Theorem \ref{suspension}}
We will first prove that we obtain transversally continuous solutions 
to the equation $\overline\partial_b u=v$ under the assumption 
that $v$ is transversally continuous and $dd^c(\tilde\sigma + s\psi)$ is positive for $s\geq 5$.  \

\subsection{Continuous solutions}
Let $\mathcal U:=\{U_{j,*}\}_{j=1}^m$ be a cover of $X$ by simply connected open sets, 
and let $\{\alpha_j\}_{j=1}^m$ be a partition of unity with respect to $\mathcal U$.  
Writing $v=\sum_j v_j:=\sum_j(\alpha_j\circ g)\cdot v$ we have reduced 
to solving $\overline\partial_b u_j=v_j$ for each $j$.   We focus on such a $v_j$
and drop the subscript $j$.  \

Let $\tilde v$ denote the form $\tilde v:=\pi^*v$ with coefficients in $\tilde L$.    
For a fixed $t\in\mathbb T$ we let $\tilde v_t$ denote the $(0,1)$-form $\tilde v_t:=v(\cdot,t)$
on the leaf $\mathbb D\times\{t\}$.  
The form satisfies 
\begin{itemize}
\item[a.] $\tilde v_t=\varphi^*\tilde v_{\phi(\varphi)(t)}$
\end{itemize}
for all $\varphi\in\Gamma$ and all $t\in\mathbb T$.  We will find transversally 
continuous solutions $\tilde u_t$ to the equations $\overline\partial\tilde u_t=\tilde v_t$
such that 

\begin{itemize}
\item[b.] $\tilde u_t=\varphi^*\tilde u_{\phi(\varphi)(t)}$
\end{itemize}
for all $\varphi\in\Gamma$ and all $t\in\mathbb T$.  We get that $u:=\pi_*\tilde u$
is well defined and solves $\overline\partial_b u=v$.  \

Let $U_{id}$ be one of the pre-images $f^*{U_*}$.   For simplicity of notation 
we assume that $0\in U_{id}$.  For any $\varphi\in\Gamma$
we let $U_{\varphi}:=\varphi(U_{id})$.    We write $\tilde v_{\varphi}:=\tilde v|_{U_\varphi\times\mathbb T}$.   Then $\tilde v=\sum_\varphi\tilde v_\varphi$.   We will solve $\overline\partial_b\tilde u_\varphi=\tilde v_\varphi$ for each $\varphi$ and then define $\tilde u:=\sum_\varphi\tilde u_\varphi$. \

For a fixed $\varphi$ and a fixed $t$ we do this as follows.  
Let $\tilde v_{\varphi,t}^*:=\varphi^*\tilde v_{\varphi,t}$, and let $\tilde u_{\varphi,t}^*$
be the $L^2(\tilde\sigma + s\psi)$-minimal solution to the equation $\overline\partial u=\tilde u_{\varphi,t}^*$ (note that $L$ is $\Gamma$-invariant).  Define $\tilde u_{\varphi,t}=\varphi_*\tilde u^*_{\varphi,t}$. We need to check that 
\begin{itemize}
\item[1.] the sum $\tilde u_t:=\sum_\varphi\tilde u_{\varphi,t}$ converges for each fixed $t$, 
\item[2.] the solutions vary continuously with $t$, and 
\item[3.] the solutions satisfy $\tilde u_t=\varphi^*\tilde u_{\phi(\varphi)(t)}$.
\end{itemize}

To show $1.$ it is enough to show that the sum converges in $L^2(\tilde\sigma)$ 
for each fixed $t$.  Let 
$c_1$ be a constant such that $\|\tilde v_{\varphi,t}^*\|_{L^2(\tilde L,\tilde\sigma + s\psi)}\leq c_1$
for all $\varphi$ and all $t$.  According to H\"{o}rmander there exists a constant 
$c_2$ such that $\|\tilde u_{\varphi,t}^*\|_{L^2(\tilde L,\tilde\sigma+s\psi)}\leq c_1\cdot c_2$
for all $\varphi$ and all $t$.  Fix $0<r<1$.   According to Lemma \ref{derivativer}
there exists a constant $c_3$ such that if $\varphi^{-1}(0)\in A_n$, then 
$e^{-\psi(\zeta)}\geq c_3\cdot 2^{n}$ for all $\zeta\in\varphi^{-1}(\mathbb D_r)$.   
According to Lemma \ref{distconst} there exists a constant $c_4$ such that 
the number of $\varphi$-s such that $\varphi^{-1}(0)\in A_n$ is no more than $c_4\cdot 2^n$.
According to Lemma \ref{derivativeupper} we have that $|\varphi'(\zeta)|\leq 2^{n+2}$
for all $\zeta\in\mathbb D$ if $\varphi(0)\in A_n.$

We get that 
\begin{align*}
\sum_{\varphi}\|\tilde u_{\varphi,t}\|_{L^2(L|_{\mathbb D_r},\tilde\sigma)} &=
\sum_\varphi
\sqrt{\underset{\mathbb D_r}{\int\int}|\tilde u_{\varphi,t}|^2e^{-\tilde\sigma}dV}\\
& =
\sum_\varphi\sqrt{\underset{\varphi^{-1}(\mathbb D_r)}{\int\int}\varphi^*[|\tilde u_{\varphi,t}|^2e^{-\tilde\sigma}dV]} \\
& = \sum_n \underset{\varphi^{-1}(0)\in A_n}{\sum}\sqrt{\underset{\varphi^{-1}(\mathbb D_r)}{\int\int}
\varphi^*[|\tilde u_{\varphi,t}|^2e^{-\tilde\sigma}dV]}\\
& \leq \sum_n \underset{\varphi^{-1}(0)\in A_n}{\sum}2^{n+2}\cdot\sqrt{\underset{\varphi^{-1}(\mathbb D_r)}{\int\int}|\tilde u^*_{\varphi,t}|^2e^{-\tilde\sigma}\cdot dV} \\
& \leq \sum_n  \underset{\varphi^{-1}(0)\in A_n}{\sum} 4c_3^{-s/2}\cdot (\frac{1}{2})^{n(s-2)/2}\sqrt{\underset{\varphi^{-1}(\mathbb D_r)}{\int\int}|\tilde u^*_{\varphi,t}|^2e^{-\tilde\sigma - s\psi}dV} \\
& \leq 4c_1c_2c_4c_3^{-s/2}\cdot\sum_n(\frac{1}{2})^{n(s-4)/2}. 
\end{align*}
This concludes the proof of $1.$   \

To show $2.$ fix $t_1\in\mathbb T$ and $\epsilon>0$.  Fix 
any integer $N\in\mathbb N$ such that $8c_1c_2c_4c_3^{-s/2}\cdot\sum_{n\geq N}(\frac{1}{2})^{n(s-4)/2}<\frac{\epsilon}{2}$.  For any $\delta>0$ we get, by the transversal continuity of $v$, that for 
all $t_2$ close enough to $t_1$ we have 
$$
\|\tilde v^*_{\varphi,t_1}-\tilde v^*_{\varphi,t_2}\|_{L^2(\tilde L,\tilde\sigma+s\psi)}\leq\delta 
\mbox { for all } \varphi \mbox{ with } \varphi^{-1}(0)\in A_n, n\leq N
$$
and so by H\"{o}rmander we get that 
$$
\|\tilde u^*_{\varphi,t_1}-\tilde u^*_{\varphi,t_2}\|_{L^2(\tilde L,\tilde\sigma+s\psi)}\leq\delta c_2
\mbox { for all } \varphi \mbox{ with } \varphi^{-1}(0)\in A_n, n\leq N.
$$
By a calculation similar to that above we see that 
$$
\sum_\varphi \|\tilde u_{\varphi,t_1}-\tilde u_{\varphi,t_2}\|_{L^2(L|_{\mathbb D_r},\tilde\sigma)}
\leq 4\delta c_2c_4c_3^{-s/2}\cdot\underset{n}{\sum}(\frac{1}{2})^{n(s-4)/2}+ \frac{\epsilon}{2}, 
$$
hence the solutions vary continuously with $t$.  \

To show $3.$ note first that $a.$ amounts to saying that 
$$
\tilde v_{\varphi,t}=\tau^*\tilde v_{(\tau\circ\varphi),\phi(\tau)(t)},
$$
for all $\varphi,\tau\in\Gamma$ and $t\in\mathbb T$, and consequently 
$$
\varphi^*\tilde v_{\varphi,t}=(\tau\circ\varphi)^*\tilde v_{(\tau\circ\varphi),\phi(\tau)(t)}.
$$
It follows that 
$$
\tilde u_{\varphi,t}=\varphi_*\tilde u^*_{\varphi,t}=\tau^*[(\tilde\tau\circ\varphi)_*\tilde u^*_{(\tau\circ\varphi),\phi(\tau)(t)}]=\tau^*\tilde u_{(\tau\circ\varphi),\phi(\tau)(t)}, 
$$
for all $\varphi,\tau\in\Gamma$ which is equivalent to $3.$ \

The proof that we obtain transversally continuous solutions with $s\geq 5$ is 
complete, and we proceed to show that the solutions are transversally 
smooth if $s$ is large enough.   

\subsection{A Lipschitz estimate}

\begin{lemma}\label{lipschitz}
Fix $0<r<1$.  Then there exists a constant $c>0$ such that  
if $\varphi\in\Gamma$ satisfies $\varphi(0)\in A_n$, and $t_1,t_2\in\mathbb T_0$, then 
$$
\|u_{\varphi,t_2}-u_{\varphi,t_1}\|_{L^2(\tilde L|_{\mathbb D_r},\tilde\sigma)}\leq c\cdot(\frac{1}{2})^{n(s-2(k+1))/2}\cdot d_0(t_1,t_2).
$$  
\end{lemma}
\begin{proof}
Note first that by the assumption that the family $v_{id,t}$ is smooth, and 
$\mathbb T$ is compact, there exists a constant $c_1>0$ such that 
\begin{itemize}
\item[1)] $\|v_{id,t_1'}-v_{id,t_2'}\|_{U_{id}}\leq c_1\cdot d_0(t_1',t_2')$,  
\end{itemize}
for all $t_1',t_2'\in\mathbb T$ (note that we are taking the sup-norm).  By possibly 
having to increase $c_1$ depending on $s$, we get the corresponding $L^2$-estimate
\begin{itemize}
\item[2)] $\|v_{id,t_1'}-v_{id,t_2'}\|_{L^2(\tilde L,\tilde\sigma + s\psi)}\leq c_1\cdot d_0(t_1',t_2')$.
\end{itemize}
Let $v_{id,t_1'}=\varphi^*v_{\varphi,{t_1}}$ and $v_{id,t_2'}=\varphi^*v_{\varphi,{t_2}}$.
By Lemma \ref{lenghtlift} we have that $d_0(t_1',t_2')\leq 2^{kn}\cdot d_0(t_1,t_2)$, 
and so by $2)$ we get that 
\begin{itemize}
\item[2')] $\|\varphi^*v_{\varphi,{t_1}}-\varphi^*v_{\varphi,{t_2}}\|_{L^2(\tilde L,\tilde\sigma + s\psi)}\leq c_1\cdot 2^{kn}\cdot d_0(t_1,t_2)$.
\end{itemize}
Since $u^*_{{\varphi,t_1}}-u^*_{\varphi,t_2}$ is the 
$L^2(\tilde\sigma+s\psi)$-minimal solution to the equation $\overline\partial u=
\varphi^*v_{\varphi,{t_1}}-\varphi^*v_{\varphi,{t_2}}$ we get that 
\begin{itemize}
\item[3)] $\|u^*_{{\varphi,t_1}}-u^*_{\varphi,t_2}\|_{L^2(\tilde L,\tilde\sigma+s\psi)}\leq c_1c_2\cdot 2^{kn}\cdot d_0(t_1,t_2)$.
\end{itemize}
A calculation similar to that above gives that  
$$
\|u_{{\varphi,t_2}}-u_{\varphi,t_1}\|_{L^2(\tilde L|_{D_r},\tilde\sigma)} \leq
c\cdot(\frac{1}{2})^{n(s-2(k+1))/2}\cdot d_0(t_1,t_2).
$$
\end{proof}

\subsection{Smoothness term by term}
\begin{lemma}\label{termbyterm}
Let $U\subset\subset\mathbb D$ and let $\tilde v_t$ be a smooth family of smooth $(0,1)$-forms 
with coefficients if $\tilde L$, each one supported in $U$.  For each $t\in\mathbb T$
let $\tilde u_t$ be the $L^2(\tilde\sigma+s\psi)$-minimal solution to the equation $\overline\partial u=\tilde v_t$.
Then $\tilde u_t$ is a smooth family of sections of $\tilde L$.  
\end{lemma}
\begin{proof}
Let $t_0\in\mathbb T$ and let $\gamma_t$ be a continuous vector field on a neighborhood $\Omega$
of $t_0$.  For each $t\in\Omega$ let $\tilde v_t^{\gamma_t}$ be the $(0,1)$-form obtained 
by differentiating $\tilde v_t$ with respect to $\gamma_t$.  Then $\tilde v_t^{\gamma_t}$ is 
a continuous family of $(0,1)$-forms.  Note that if $\gamma_{t,s}$ is 
a continuous family of vector fields parametrized by $s$, then $\tilde v_t^{\gamma_{t,s}}$
is continuous also in $s$.  \

For each $t$ let $\tilde u^{\gamma_t}_t$ be the $L^2(\tilde\sigma+s\psi)$-minimal solution 
to the equation $\overline\partial u=\tilde v_t^{\gamma_t}$.   Then $\tilde u^{\gamma_t}_t$
is a continuous family of sections.   We claim that at any point $(\zeta,t)$
we have that $\gamma_t(\tilde u_t)(\zeta)=\tilde u^{\gamma_t}_t(\zeta)$.  \

Let $\gamma:(-1,1)\rightarrow\mathbb T$ be a smooth curve with $\gamma(0)=t$
and $\gamma'(0)=\gamma_t$.  Let $\delta\gamma$ denote the point $\gamma(\delta)$.
Clearly 
$$
\overline\partial(\frac{\tilde u_{t+\delta\gamma}-\tilde u_t}{\delta})=\frac{\tilde v_{t+\delta\gamma}-v_t}{\delta},
$$
and it is also the $L^2(\tilde\sigma+s\psi)$-minimal solution.  Since the right-hand side 
converges uniformly it follows by H\"{o}rmanders estimate that 
$$
\frac{\tilde u_{t+\delta\gamma}-\tilde u_t}{\delta}
$$
converges uniformly to $\tilde u_t^{\gamma_t}$. \

Finally we could check the continuity of the disk-derivatives by the same method, or 
we could produce another solution whose disk-derivatives vary continuously
using the Cauchy-formula, and then conclude by the Cauchy-estimates.   

\end{proof}

\subsection{Smoothness of the sum}

Fix $t\in\mathbb T$.  We will show that $u$ is smooth near $\mathbb D\times\{t\}$. 
Using local coordinates on $\mathbb T$ we may assume that the point $t$ is the origin
in $\mathbb R^n$ and that the metric is the usual one (since everything is comparable).  Let $\gamma$ be a vector of norm one in $\mathbb R^n$.  
We need
to estimate 
$$
\|\frac{\tilde u_{t+\delta\gamma}-\tilde u_t}{\delta}\|_{L^2(\tilde L|_{\mathbb D_r},\tilde\sigma)},
$$
where $\tilde u_t$ is defined by 1. above.  \

Using Lemma \ref{lipschitz} and following the arguments for continuity we see that 
\begin{align*}
\|\frac{\tilde u_{t+\delta\gamma}-\tilde u_t}{\delta}\|_{L^2(\tilde L|_{\mathbb D_r},\tilde\sigma)}
& \leq\underset{\varphi^{-1}(0)\in A_n,n\leq N}{\sum}\|\frac{\tilde u_{\varphi,t+\delta\gamma}-\tilde u_{\varphi,t}}{\delta}\|_{L^2(\tilde L|_{\mathbb D_r},\tilde\sigma)}\\
& + \underset{n>N}{\sum}c\cdot(\frac{1}{2})^{n(s-2(k+2))/2},
\end{align*}
for any $N\in\mathbb N$.   Note also that all but a finite number of $(\tilde u_{t+\delta\gamma}-\tilde u_t)$-s are holomorphic on $\mathbb D_r$.  Smoothness 
follows by Lemma \ref{termbyterm}, and the proof of Theorem \ref{suspension}
is complete. 
$\hfill\square$

\begin{remark}
It is also possible to solve $\overline\partial_b$ on suspensions over tori.  In that 
case it can be done more explicitly by following the above procedure, but using 
a weighted Cauchy integral formula for solving $\overline\partial$:
$$
u(z):=\frac{-1}{\pi\cdot z^k}\underset{\mathbb C}{\int\int}\frac{v(\zeta)\zeta^k}{\zeta-z}d\zeta\wedge d\overline\zeta,
$$
on the universal cover $\mathbb C\rightarrow\mathbb C/\Gamma$. 
\end{remark}

\section{Proof of Theorem \ref{furstenberg}}\label{embedding}

In \cite{FornaessSibonyWold} we proved with Sibony that there 
exists a smooth suspension over a compact Riemann surface of 
genus two, which is a minimal lamination supporting
uncountably many extremal closed laminated currents which are mutually singular. 
The lamination is of real transverse dimension two.  It therefore 
suffices to prove the following: 

\begin{theorem}
Let $g:Y\rightarrow X$ be a $\mathcal C^1$-smooth suspension over a 
compact Riemann surface $X$ of genus $g_X\geq 2$.   Then 
$Y$ is projective.  
\end{theorem}
\begin{proof}
We need to find a line bundle $L\rightarrow Y$ where we can find enough 
sections to separate points and to have non-vanishing differentials.  We will do this by 
constructing local sections and solving $\overline\partial_b$ with singular weights.  \

Fix $s$ according to Theorem \ref{suspension}, and define 
$$
\tilde\omega (\zeta):=\frac{1}{(1-|\zeta|^2)^2}dV, \zeta\in\mathbb D.
$$
Since $\tilde\omega$ is $\Gamma$-invariant it defines a volume form $\omega_*:=f_*\tilde\omega$ on $X$.  \

Note that there exists a constant $c_1>0$ such that the following holds 
\begin{itemize}
\item[a.] for every pair $x_1,x_2\in X$ there exists a function $\tau_*\in L^2_{loc}(X)$
with $\nu(\tau_*,x_j)=2$ for $j=1,2$, and $dd^c\tau_*\geq -c_1\cdot\omega_*$. 
\end{itemize}  
Here $\nu$ denotes the Lelong number.  The claim follows by compactness of $X$ and the construction of such a $\tau$ for fixed $x_1,x_2$ (see 2.2.1 of \cite{FornaessSibonyWold} for details).  \

Since $X$ is projective there exists a line bundle $L_*\rightarrow X$
with a smooth metric $\sigma_*$ such that 
\begin{itemize}
\item[b.] $dd^c\sigma_*\geq c_2\cdot\omega_*$, for some $c_2>0$.
\end{itemize}  
At this point we fix $N\in\mathbb N$ such that 
\begin{itemize}
\item[c.] $N\cdot c_2>c_1 + 4s$.
\end{itemize}
Define $\tilde\sigma:=f^*\sigma_*$ and $\sigma:=g^*\sigma_*$, and for any such $\tau_*$ we define $\tilde\tau:=f^*\tau_*$ and $\tau:=g^*\tau_*$. By a. and b. 
we have that 
$$
dd^c(k\cdot\tilde\sigma + \tilde\tau + s\psi)\geq (k\cdot c_2 - c_1 - 4s)\cdot\tilde\omega
, 
$$
and so by $c.$ we have that for all $k\geq N$ we may solve $\overline\partial_b$
for $L^2$-sections of $L^{\otimes k}$ over $Y$, and the solutions will 
be in $L^2_{loc}(k\sigma + \tau)$ as well as being $\mathcal C^1$-smooth on 
the total space.  The main point is that this will force the solutions to vanish 
to order two in the leaf direction along the transversals over the points $x_1$ and 
$x_2$.  We now sketch the steps to produce sufficiently many sections of $L^{\otimes k}$ to produce an embedding. 
\begin{itemize}
\item[i.] \emph{Non-vanishing differentials in the leaf direction:}  Here we 
can use sections of the bundle $L^{\otimes k}_*$.  Given a point $x\in X$ let 
$\xi_1$ be a smooth section of  $L^{\otimes k}_*$, holomorphic near $x$, 
which in local coordinates ($x=0$) looks like $\xi_1(z)=z + O(|z|^2)$.  Let 
$\xi_2$ be a section that looks like $\xi_2(z)=1+O(|z^2|)$.   Let $v_j:=\overline\partial\xi_j$, solve $\overline\partial u_j=v_j$ with metric which 
is singular at $x$, and define $s_j:=\xi_j-u_j$.  The quotient 
$s_1/s_2$ has a non-vanishing differential in the leaf direction on a full neighborhood 
of the transversal $\mathbb T_x$.  By compactness we cover all of $X$.   
\item[ii.] \emph{Separate points over different points in the base:}  Use a similar construction as i. to separate transversals $\mathbb T_{x_j}$ and $\mathbb T_{x_i}$ for $x_i\neq x_j$.  Use compactness and i. to cover everything.  

\item[iii.] \emph{The transversal direction:}  For any given $x\in X$ start with 
smooth sections $\xi_1,...,\xi_m$ of $L^{\otimes k}|_{\mathbb T_x}$ providing 
an embedding of $\mathbb T_x$ into projective space.  Extend the sections 
$\xi_j$ constantly along leaves near $\mathbb T_x$ and use a cut-off function 
on the base to extend each $\xi_j$ to a section of $L^{\otimes k}$.   Define 
$v_j:=\overline\partial_b\xi_j$, solve $\overline\partial_b u_j=v_j$ with a weight 
which is singular along $\mathbb T_x$, and define $s_j:=\xi_j-u_j$.   Then 
each $s_j$ will have the same differential as $\xi_j$ along $\mathbb T_x$, and 
since they are all $\mathcal C^1$-smooth, they provide an embedding of 
all transversals near $\mathbb T_x$.  By compactness we cover all transversals. 
\end{itemize}
Note that it was only for iii. we used Theorem \ref{suspension}.

\end{proof}

\section{Proof of the Main Theorem}

Let 
$$
\phi_{\alpha}=(z_\alpha,t_\alpha):U_{\alpha}\rightarrow\mathbb D\times\mathbb T
$$
be a flow-box.  Let $\mathbb T_0$ denote the transversal $\phi_{\alpha}^{-1}(\{0\}\times\mathbb T)$.  By the continuity of the Kobayashi metric 
we may choose a constant $r>0$ such that 
\begin{itemize}
\item[a.] if $\gamma:[0,1]\rightarrow\mathcal L_t$ is a smooth curve such that,
$\gamma(0),\gamma(1)\in U_\alpha, z_\alpha(\gamma(0))=z_\alpha(\gamma(1))=0$
and $t_\alpha(\gamma(0))\neq t_\alpha(\gamma(1))$,  
then the Kobayashi length of $\gamma$ is greater than $r$.  
\end{itemize}
Simply let 
$r$ be the infimum of the Kobayashi radii of the plaques $\mathcal L_{\alpha,t_\alpha(t)}$ for $t\in\mathbb T_0$.  
Similarly we may choose $r$ such that 
\begin{itemize}
\item[b.] if $\gamma:[0,1]\rightarrow\mathcal L_t$ is a 
non-trivial curve with $\gamma(0)=\gamma(1)=t\in\mathbb T_0$, 
then the length of $\gamma$ is greater than or equal to $r$.
\end{itemize}
Choose $0<r'<r$, and we get that  
\begin{itemize}
\item[c.] the K-disk $\triangle_{K,r'}(t)$ in $\mathcal L_t$ of 
radius $r'$ centered at $t$ is contained in $\mathcal L_{\alpha,t_\alpha(t)}$ for all $t\in\mathbb T_0$. 
\end{itemize}
By compactness it is enough to solve $\overline\partial_b u=v$ when $v$ is compactly 
supported in 
$\underset{t\in\mathbb T_0}{\cup}\triangle_{K,r'}(t). $

\

Let $\tilde\xi(t)$ be the vector field $\phi_\alpha^*(\partial/\partial\zeta|_{\zeta=0})$, 
and let $\xi(t)$ be the corresponding vector field normalized by the Kobayashi metric.  
For each $t\in\mathbb T_0$ let $f_t:\mathbb D\rightarrow\mathcal L_t$ be the universal covering map with 
$f_t(0)=t$ and $f_t'(0)=\xi(t)$.
Let $L_t$
denote the line bundle $L_t:=f_t^*L$ over $\mathbb D$.  Let $\sigma_t$
denote the metric $f_t^*\sigma$.  By continuity of the Kobayashi metric we have that 
$$
dd^c\sigma_t(\zeta)=\frac{g_t(\zeta)}{(1-|\zeta|^2)^2}dV, 
$$
where $g_t$ is bounded from below $g_t(\zeta)\geq c>0$ independently of $t$ and $\zeta$.
By passing to a power of $L$ we may assume that $g_t(\zeta)\geq c>20$, and so 
$dd^c(\sigma_t + 5\psi)$ is strictly positive independently of $t$, \emph{i.e.}, we 
may solve $\overline\partial$ with estimates using the metric $\sigma_t+5\psi$.  
Let $v_t$
denote the form $v_t:=f_t^*v$.   \

For each $t\in\mathbb T$ let $E_t$ denote the discrete set of points $E_t:=\{f_t^{-1}(\mathbb T_0)\}$, and note that  
the Kobayashi distance between any two points in $E_t$ is greater than $r$.  
For each point $\zeta\in E_t$ let $U_{t,\zeta}$ denote the connected component 
of $f_t^{-1}(\mathcal L_{\alpha,t_\alpha(f(\zeta))})$ containing $\zeta$.  Let $v_{t,\zeta}:=v_t|_{U_\zeta}$
and note that $v_{t,\zeta}$ is compactly supported in $U_{t,\zeta}$.  In fact,  
$v_{t,\zeta}$ is compactly supported the K-disk of radius $r'$ centered at $\zeta$.  \

For each $n=0,1,2,...$, let $E_{t,n}$ denote the set 
$$
E_{t,n}=\{\zeta\in E_t:1-(\frac{1}{2})^n\leq |\zeta|<1-(\frac{1}{2})^{n+1}\}.
$$

By Lemma \ref{distconst} the following holds: 

\begin{itemize}
\item[d.] there exists a constant $c>0$ such that 
$\sharp(E_{t,n})\leq c\cdot 2^n$ for all $n$, and for all $t$.    
\end{itemize}

For each $\zeta\in E_t$ let $\varphi_{t,\zeta}$ denote \emph{the} element 
$\varphi_{t,\zeta}\in\Aut_{hol}\mathbb D$ with $\varphi_{t,\zeta}(0)=\zeta$
and $(f_t\circ\varphi_{t,\zeta})'(0)=\xi(f_t(\zeta))$.  Let $v_{t,\zeta}^*$ 
denote the form $\varphi_{t,\zeta}^*v_{t,\zeta}$, let $u_{t,\zeta}^*$
denote the $L^2(\varphi_{t,\zeta}^*(\sigma_t) + 5\psi)$-minimal solution 
to the equation $\overline\partial u=v^*_{t,\zeta}$, and 
let $u_{t,\zeta}:=(\varphi_{t,\zeta})_*u^*_{t,\zeta}$.  We define 
\begin{itemize}
\item[$(*)$] $u_t:=\underset{\zeta\in E_t}{\sum}u_{t,\zeta}$, 
\end{itemize}
and then finally 
\begin{itemize}
\item[$(**)$] $u:=(f_t)_*(\underset{\zeta\in E_t}{\sum}u_{t,\zeta})$ on $\mathcal L_t$.  
\end{itemize}
We need to check that 
\begin{itemize}
\item[i)] the sum $(*)$ converges for each $t\in\mathbb T_0$, 
\item[ii)] the push-forward $(**)$ is well defined, and 
\item[iii)] the solutions vary continuously between leaves.  
\end{itemize}

$i)$ and $iii)$ are proved essentially as in the special case of a suspension over a genus 
$g$ surface, $g\geq 2$, but we need Theorems \ref{kmetric} and \ref{continuous}.  
For convergence we need to note that, due 
to continuity of the Kobayashi metric, there exists a constant $c_1>0$
such that 
\begin{itemize}
\item[e.] $\|v^*_{t,\zeta}\|_{L^2(L,\sigma_t^*+5\psi)}\leq c_1$, 
\end{itemize}
for all $\zeta$ and all $t$.  A similar calculations as before then gives, for a fixed $0<r<1$, that,    
\begin{itemize}
\item[f.] $\underset{\zeta\in E_t}{\sum}\|u_{t,\zeta}\|_{L^2(L_t|_{\mathbb D_r},\sigma_t)}\leq C\cdot\underset{n}{\sum}\underset{\zeta\in E_{t,n}}{\sum}(\frac{1}{2})^{n(s-2)/2}
\cdot \|v^*_{t,\zeta}\|_{L^2(L,\sigma_t^*+5\psi)}$, 
\end{itemize}
where the constant $C$ is independent of $t$.   Then d. and e. gives convergence.  \

For continuity note first that, by the same calculation as the one leading to f., we get for any $N\in\mathbb N$ that
\begin{itemize}
\item[g.] $\underset{n\geq N, \zeta\in E_{t,n}}{\sum}\|u_{t,\zeta}\|_{L^2(L_t|_{\mathbb D_r},\sigma_t)}\leq C\cdot\underset{n\geq N}{\sum}\underset{\zeta\in E_{t,n}}{\sum}(\frac{1}{2})^{n(s-2)/2}
\cdot \|v^*_{t,\zeta}\|_{L^2(L,\sigma_t^*+5\psi)}$.
\end{itemize}
So using d. and e. it follows that for any $\epsilon>0$ there exists an $N\in\mathbb N$ such that 
\begin{itemize}
\item[h.] $\underset{n\geq N, \zeta\in E_{t,n}}{\sum}\|u_{t,\zeta}\|_{L^2(L_t|_{\mathbb D_r},\sigma_t)}<\epsilon$, 
\end{itemize}
for all $t$.  So proving convergence is reduced to proving that finite sums converge.  Since 
 $f_{t}\underset{t \rightarrow t_0}{\longrightarrow} f_{t_0}$
u.o.c. in $\mathbb D$ (Theorem \ref{kmetric}), this is easily reduced to showing that 
$u_{t,0}\rightarrow u_{t_0,0}$ as $t\rightarrow t_0$.  This is the content of 
Theorem \ref{continuous}.   \

\medskip

$ii)$ 
Let $t_1, t_2\in\mathbb T_0$ both be contained in the same leaf $\mathcal L_t$ (we allow them to be the same point).    
Let $\psi$ be any element $\psi\in Aut_{hol}\mathbb D$ such that $f_{t_2}=f_{t_1}\circ\psi$.  We need to show that $u_{t_2}=\psi^*u_{t_1}$.

Let $\zeta\in E_{t_2}$ and note that $v_{t_2,\zeta}=\psi^*v_{t_1,\psi(\zeta)}$, and that

$\varphi_{t_2,\zeta}=\psi^{-1}\circ\varphi_{t_1,\psi(\zeta)}$. 
We have that $v_{t_2,\zeta}^*=\varphi_{t_2,\zeta}^*v_{t_2,\zeta}$, and so 
$$
v_{t_2,\zeta}^*=(\psi^{-1}\circ\varphi_{t_1,\psi(\zeta)})^*(\psi^*v_{t_1,\psi(\zeta)})
= \varphi_{t_1,\psi(\zeta)}^*v_{t_1,\psi(\zeta)}=v^*_{t_1,\psi(\zeta)}.
$$
We get 
$$
u_{t_2,\zeta}=(\varphi_{t_2,\zeta})_*u_{t_1,\psi(\zeta)}^*=\psi^{-1}_*(\varphi_{t_1,\psi(\zeta)})_*u_{t_1,\psi(\zeta)}^*=\psi^*u_{t_1,\psi(\zeta)}.  
$$
From this we see that $u_{t_2}=\psi^*u_{t_1}$.  This shows that $u_{t}$ is 
well defined on the quotient and that it is independent of the choice of 
transversal point.

\medskip

\noindent John Erik Forn\ae ss\\
University of Michigan\\
Mathematics Department\\
2074 East Hall, 530 Church Street\\
Ann Arbor, Michigan 48109-1043\\
USA\\
fornaess@umich.edu\\

\noindent Erlend Forn\ae ss Wold\\
Universitetet i Oslo\\
Matematisk Institutt\\
Postboks 1053, Blindern\\
NO-0316 Oslo\\
Norway\\
erlendfw@math.uio.no\\


\begin{thebibliography}{10}

\bibitem{Brody} Brody, Robert;
Compact manifolds in hyperbolicity. 
\textit{Trans. Amer. Math. Soc.} \bf 235\rm, (1978), 213--219.

\bibitem{Candel}
Candel, Alberto; Uniformization of surface laminations. 
\textit{Ann. Sci. \'Ecole Norm. Sup.} (4) \bf 26\rm,  (1993), no. 4, 489--516.


\bibitem{Cao-Shaw-Wang} Cao, Jiangguo; Shaw, Mei-Chi; Wang, Lihe; Estimates for the $\overline{\partial}-$Neumann problem and nonexistence of $\mathcal C^2$ Levi-flat hypersurfaces in $\mathbb{CP}^n.$ \textit{Math. Z.} \bf 248\rm,  (2004), no. 1, 183--221. Erratum \textit{Math. Z.} \bf 248\rm, (2004), no. 1, 223--225.

\bibitem{Chai} Chai, Xiaoai; Properties of solutions of differential equations on laminations.
REU project at the University of Michigan, 2010.

\bibitem{Deroin}
Deroin, Bertrand;
Laminations dans les espaces projectifs complexes. 
\textit{J. Inst. Math. Jussieu} \bf 7, \rm (2008), no. 1, 67--91.

\bibitem{DihnNguyenSibony}
Entropy for hyperbolic Riemann surface laminations I. 
http://arxiv.org/pdf/1105.2307

\bibitem{Forstneric} Forstneri\v{c}, Franc; 
Noncritical holomorphic functions on Stein manifolds. 
\textit{Acta Math.} \bf 191\rm,  (2003), no. 2, 143--189. 

\bibitem{Ghys}
Ghys, \'{E}tienne;
Gauss-Bonnet theorem for 2-dimensional foliations. 
\textit{J. Funct. Anal.} \bf 77\rm , (1988), no. 1, 51--59.

\bibitem{Ghys2}
Ghys, \'{E}tienne;
Laminations par surfaces de Riemann.
Dynamique et gŽomŽtrie complexes (Lyon, 1997), ix, xi, 49--95, Panor. Synthses, 8, Soc. Math. France, Paris, 1999.

\bibitem{Gromov}
Gromov, Misha;
Topological invariants of dynamical systems and spaces of holomorphic maps. I. 
\textit{Math. Phys. Anal. Geom.} \bf 2,\rm  (1999), no. 4, 323--415. 


\bibitem{FornaessSibony}
Forn\ae ss, John Erik,  Sibony, Nessim; Harmonic currents of finite energy and laminations. 
\textit{Geom. Funct. Anal.} \bf 15\rm, (2005), no. 5, 961--1003.

\bibitem{FornaessSibonyWold}
Forn\ae ss, John Erik, Sibony, Nessim, Wold, Erlend Forn\ae ss; 
Some examples of minimal laminations and associated currents.  
To appear in Math. Z. 


\bibitem{Iordan-Matthey} Iordan, Andrei; Matthey, Fanny; R\'egularit\'e de l'operateur $\overline{\partial} $ et th\'eor\`eme de Siu sur la non-existence d'hypersurfaces Levi plates dans l'espace projectif. \textit{C. R. Math. Acad. Sci. Paris} \bf 346\rm, (2008), no. 7-8, 395--400.

\bibitem{Ohsawa-Sibony} Ohsawa, Takeo; Sibony, Nessim; K\"{a}hler identity on Levi flat manifolds and application to embedding. \textit{Nagoya Math. J.} \bf 158\rm, (2000), 87--93.

\bibitem{Ro}
Royden, Halsey Lawrence Jr; The extension of Regular Holomorphic Maps. 
\textit{Proc. Am. Math. Soc.} \bf 43\rm, (1974),  306--310.

\bibitem{Siu} Siu, Yum-Tong; $\overline{\partial}-$regularity for weakly pseudoconvex domains in compact Hermitian symmetric spaces with respect to invariant metrics. \textit{Ann. of Math. (2)} \bf 156\rm, (2002), no. 2, 595--621.

\bibitem{Verjovsky}
Verjovsky, Alberto;
A uniformization theorem for holomorphic foliations. The Lefschetz centennial conference, Part III (Mexico City, 1984), 233--253, 
\textit{Contemp. Math.}, \bf 58\rm, III, Amer. Math. Soc., Providence, RI, 1987. 

\end{thebibliography}
 \end{document}